\pdfoutput=1
 \documentclass{article}
\usepackage{graphicx} 

\usepackage{amsthm,amsfonts,amssymb,amsmath,epsf, verbatim}

\newtheorem{theorem}{Theorem}
\newtheorem{lemma}[theorem]{Lemma}

\newtheorem{proposition}[theorem]{Proposition}

\newtheorem{conjecture}[theorem]{Conjecture}

\usepackage{graphicx} 

\title{On 2-Near Perfect Numbers}
\author{{\bf Vedant Aryan}\\ varyan24@students.hopkins.edu \\ {\bf Dev Madhavani}\\ devmadhavani@college.harvard.edu \\ {\bf Savan Parikh}\\ savan.parikh@yale.edu \\ {\bf Ingrid Slattery}\\ ingrid.slattery@yale.edu \\{\bf Joshua Zelinsky} \\jzelinsky@hopkins.edu \\  }
\date{}

\begin{document}

\maketitle

\begin{abstract} Let $\sigma(n)$ be the sum of the positive divisors of $n$. A number $n$ is said to be 2-near perfect if $\sigma(n) = 2n +d_1 +d_2 $, where $d_1$ and $d_2$ are distinct positive divisors of $n$. We give a complete description of those $n$ which are 2-near perfect and of the form $n=2^k p^i$ where $p$ is prime and $i \in \{1,2\}$. We also prove related results under the additional restriction where $d_1d_2=n$.
\end{abstract}
\section{Introduction}
A \textit{perfect number} is a positive integer that is equal to the sum of its proper positive
divisors. Equivalently, a perfect number is an $n$ such that $\sigma(n)=2n$, where $\sigma(n)$ is the sum of all positive divisors of $n$. Perfect numbers have been studied since antiquity. 
The idea of perfect numbers has been generalized in a variety of ways. A classic generalization is the notion of a \textit{multiply perfect number}, defined as a number $n$  which satisfies $\sigma(n)=mn$ for some other integer $m$. Sierpi\~{n}ski\cite{Sierpinski} introduced the term \textit{pseudoperfect} number to mean a number $2n$ such that $n$ is the sum of some subset of its divisors. Pollack and Shevelev \cite{PS} separated the pseudoperfect numbers into separate types by introducing the idea of a $s$-near perfect number. A number $n$ is $s$-near perfect if $2n$ is the sum of all its positive divisors excepting $s$ of them. For example, while 12 is not perfect, it is 1-near perfect, since $1+2+3+6+12=2(12)$. Here the divisor which has been removed from the set is 4. We will refer to divisors removed from the set as {\it{omitted divisors}}. Note that any pseudoperfect number is $s$-near perfect for some $s$, and one can think of perfect numbers as $0$-near perfect numbers. In some sense, multiply perfect numbers are a multiplicative generalization of perfect numbers, while $s$-near perfect numbers are a more additive generalization.

A number $n$ is said to be abundant if it satisfies $\sigma(n) > 2n$. If $n$ is $s$-near perfect for some $s>0$, then $n$ must be abundant. However, it is possible for a number to be abundant while not being $s$-near perfect for any $s$. An example is 70, where $\sigma(70)=144$. Numbers which are abundant but not $s$-near perfect for any $s$ are said to be {\it{weird}}. A classic open problem is whether there are any odd weird numbers. 

 In  addition to the more general notion of $s$-near perfect numbers, Pollack and Shevelev \cite{PS} also used the term near perfect number to mean 1-near perfect number. Another classic open problem is whether there is any $n$ such that $\sigma(n) = 2n+1$. Such numbers are called quasiperfect numbers.  Note that any quasiperfect number is a $1$-near perfect number with omitted divisor $1$.

 Pollack and Shevelev constructed the following three distinct families of 1-near perfect numbers: 

\begin{enumerate}
\item $2^{t-1}(2^t-2^k-1)$ where $2^t-2^k-1$ is prime. Here $2^k$ is the omitted divisor.
\item $2^{2p-1}(2^p-1)$ where $2^p-1$ is prime. Here $2^p(2^p-1)$ is the omitted divisor.
\item $2^{p-1}(2^p-1)^2$ where $2^p-1$ is prime. Here $2^p-1$ is the omitted divisor.
\end{enumerate}

Subsequent work by Ren and Chen \cite{RC} showed that all near perfects with two distinct prime factors  must be either 40, or one of the the three families above.

The only known near perfect odd number is $173369889=(3^4)(7^2)(11^2)(19^2).$  Tang, Ma, and Feng \cite{TMF} showed that this is the only odd near perfect number with four or fewer distinct prime divisors. Cohen, Cordwell, Epstein, Kwan, Lott, and Miller proved general asymptotics for $s$-near perfect numbers for $s \geq 4.$ Recent work by Hasanalizade \cite{EH} gave a partial classification of near perfect numbers which are also Fibonacci or Lucas numbers.

 Li and Liao \cite{LL} classified all even near perfects of the form $2^a p_1p_2$ where $p_1$ and $p_2$ are distinct primes. 

The main results of this paper are twofold. First, we give a complete description of $2$-near perfects of the form $2^kp$ or $2^kp^2$ where $p$ is prime.  Second, we use these characterizations to introduce a closely related notion of strongly 2-near perfect numbers, and give a characterization of those of the form $2^kp$. 

In particular, we have the following  two main results.

\begin{theorem} Assume $n$ is a $2$-near perfect number with omitted divisors $d_1$ and $d_2$. Assume further that $n=2^k p$ where $p$ is prime and $k$ is a positive integer. Then one must have, without loss of generality, one of four situations. \label{classification of 2 near perfect of form power of two times a prime}
\begin{enumerate}
    \item $p=2^k-1$. Here we have $d_1=1$ and $d_2=p$.
    \item $p=2^{k+1} -2^a -2^b-1$ for some $a, b \in \mathbb{N}$. Here  $d_1=2^a$ and $d_2=2^b$.
    \item $p=\frac{2^{k+1}-2^a-1}{1+2^b}$ for some $a, b \in \mathbb{N}$. Here  $d_1 = 2^a$ and $d_2 = 2^bp$. 
    \item $p=\frac{2^{k+1}-1}{1+2^a+2^b}$ for some $a, b \in \mathbb{N}$. Here $d_1 =2^ap$ and $d_2=2^bp$.
\end{enumerate}
\end{theorem}

\begin{theorem} Assume that $n$ is a $2$-near perfect number with omitted divisors $d_1$ and $d_2$. Assume further that $n=2^kp^2$ where $p$ is prime. Then $n \in \{18, 36,200\}$. \label{Second main result power of 2 times square of a prime}
\end{theorem}

We recall the following basic facts about $\sigma(n)$ will be useful throughout: 
\begin{lemma} The function $\sigma(n)$ has the following properties: \label{basic sigma properties}
\begin{enumerate}
    \item $\sigma(n)$ is multiplicative. That is, $\sigma(ab) =\sigma(a)\sigma(b)$ whenever $a$ and $b$ are relatively prime.
    \item For a prime $p$, $\sigma(p^k) = p^k + p^{k-1} \cdots +1 = \frac{p^{k+1}-1}{p-1}.$
\end{enumerate}

\end{lemma}

\section{Proof of Theorem \ref{classification of 2 near perfect of form power of two times a prime}. }

Let us now prove Theorem \ref{classification of 2 near perfect of form power of two times a prime}.
\begin{proof}
Assume 
we have a 2-near perfect number of the form $n=2^{k}p$ with two omitted divisors $d_1, d_2$, $d_1\neq d_2$ and odd prime $p$. Because $n$ is near perfect, we have that:
\begin{equation}
    \sigma(n)=2n+d_1+d_2.
\end{equation}

Using Lemma \ref{basic sigma properties}, we then have:

\begin{equation}
    \sigma(n)=\sigma(2^{k}p)=(2^{k+1}-1)(p+1)
\end{equation}
So, setting equation (1) equal to equation (2), we have:
\[(2^{k+1}-1)(p+1)=2n+d_1+d_2=2^{k+1}p+d_1+d_2,\] and hence 
\begin{equation}
    p=2^{k+1}-1-d_1-d_2. \label{p equals in 2 to k p equation}
\end{equation}
Because $p$ is odd, we have that $2^{k+1}-1-d_1-d_2$ is odd. Since $2^{k+1}-1$ is always odd, we have that $-(d_1+d_2)$ must be even. If $-(d_1+d_2)$ is even, $d_1$ and $d_2$ must be of the same parity. We thus need to consider two situations: where $d_1, d_2$ are both odd, and where they are both even. We will call the first situation Case 1, and we shall separate the second situation, Case 2, into three separate subcases without loss of generality.
\\
\\
Case 1: In this case, $d_1, d_2$ are both odd.
\\
\\
The only odd divisors of $n$ are $1$ and $p$, so we can, without loss of generality, set $d_1=1$ and $d_2=p$ to find:
\[p=2^{k+1}-1-1-p,\] and hence
\[p=2^{k}-1.\]
Thus, our first family of 2-near perfect numbers correspond to Mersenne primes and have the form $2^k(2^k-1)$ (twice an even perfect number). 
\\

We now consider the situation where where $d_1, d_2$ are both even. We shall break this down into three subcases, depending on the types of values for $d_1$ and $d_2$.


Case 2.1: In this case we have $d_1=2^a, d_2=2^b,$ where $0<a<b\leq k$
\\
We can use in our definitions of $d_1$ and $d_2$ in Equation (3) to find:
\[p=2^{k+1}-2^a-2^b-1\]
This is our second family of 2-near perfect numbers.
\\
\\
Case 2.2: In this case, $d_1, d_2$ are both even, and $d_1=2^a, d_2=2^{b}p,$ and $ a, b \in (0, k]$
\\
\\
We use a similar strategy, and plug in our definitions of $d_1,d_2$ into Equation (\ref{p equals in 2 to k p equation}):
\[p=2^{k+1}-2^a-2^{b}p-1,\] and so
\[p(1+2^b)=2^{k+1}-2^a-1,\] which becomes
\[p=\frac{2^{k+1}-2^a-1}{1+2^b}.\]
This is our third family of 2-near perfect numbers. 
\\
\\
Case 2.3: In this case, $d_1, d_2$ are both even and we have $d_1=2^{a}p, d_2=2^{b}p,$ and $ 0 < a < b \leq k$
\\
\\
Using the same technique, equation (3) tells us:
\[p=2^{k+1}-1-2^{a}p-2^{b}p,\]
and hence, \[p(1+2^a+2^b)=2^{k+1}-1,\] which implies that
\[p=\frac{2^{k+1}-1}{1+2^a+2^b}.\]
This is the fourth and final family of 2-near perfect numbers. 
\end{proof}

\section{Proof of Theorem \ref{Second main result power of 2 times square of a prime}}

One major technique we will use is what we call the {\it discriminant sandwich} method: we show that a given Diophantine equation has only a restricted set of possible solutions. We do so by showing that the equation is a quadratic equation in one variable, and thus in order to have integer valued solutions,  the discriminant must be a perfect square. However, we will show that the discriminant must, except in a limited set of cases, be shown to be strictly between two consecutive perfect squares, and thus aside from those situations, we have no solution. Discriminant sandwiching will be used extensively in what follows.

\begin{lemma} Let $a$ and $k$ be positive integers such that $D= 2^{2k+2}+2^{k+2}-2^{a+2}-7$. If $0 \leq a \leq k$ and $D$ is a perfect square, then $k=a=1$. \label{Case I Lemma 1}
\end{lemma}

\begin{proof} Let us assume that $D$ is a perfect square. For all $a$, note that
$$2^{2k+2}+2^{k+2}-2^{a+2}-7 < 2^{2k+2} +2^{k+2}+2 = (2^{k+1}+1)^2$$
Thus, if we have 
\begin{equation}
D = 2^{2k+2}+2^{k+2}-2^{a+2}-7 > 2^{2k+2} = (2^{k+1})^2
\end{equation} then the quantity in question cannot be a perfect square because it is sandwiched between two consecutive perfect squares. So we must have that 
\[2^{2k+2}+2^{k+2}-2^{a+2}-7 \leq 2^{2k+2}\] and therefore 
\begin{equation}
    2^{k+2}-2^{a+2} \leq 7. \label{Case I Lemma inequality}
\end{equation}
If $k > a \geq 1$, then from Equation \ref{Case I Lemma inequality} we have that $2^{k+2}-2^{k+1} \leq 2^{k+2}-2^{a+2} \leq 7 $. Thus, 
\[2^{k+2}-2^{k+1} = 2^{k+1} \leq 7,\] which implies that $k \leq 1$. 

However, given the conditions for this case, no solutions are possible.
\\
\\
Now, consider the case when $k=a\geq 1$. In this case, we have
\[D=2^{2k+2}+2^{k+2}-2^{k+2}-7=2^{2k+2}-7.\]
Given this, note that:
\[2^{2k+2}-7<2^{2k+2}=(2^{k+1})^2\]
Using the same logic as earlier, we see that if
\[D>(2^{k+1}-1)^2,\]
then $D$ will be sandwiched between two consecutive perfect squares, and thus will not be a square itself. Thus, we can assume:
\[2^{2k+2}-7\leq (2^{k+1}-1)^2=2^{2k+2}-2^{k+2}+1,\] which implies that $k\leq 1$.
The bounds for this case require $k\geq 1$, so the only solution possible is $(a,k)=(1,1)$.
\end{proof}

Essentially, Lemma \ref{Case I Lemma 1} is the sandwiching part of the discriminant sandwich we will use in the proof of Proposition \ref{n =18 situation} below. 

\begin{lemma} Let $b$ and $k$ be non-negative integers and $p$ be an odd number such that \begin{equation} (2^{k+1}-1)(p^2+p+1)= 2^{k+1}p^2 + 2^bp +1. \label{Copy of Equation for Case III, j=1, general form for Lemma 1}
\end{equation}  then $p|2^k-1$, and $p+1|2^b-2$. \label{first Lemma for Case III, a=0, j=1}
\end{lemma}
\begin{proof} Assume one has a solution to Equation \ref{Copy of Equation for Case III, j=1, general form for Lemma 1}. Then, if we take the equation modulo $p$, we get that 
\[2^{k+1}-1 \equiv 1  \mod p,\]
and hence $p|2^{k+1}-2= 2(2^{k}-1)$. Since $p$ is odd, we have $p|2^{k}-1$.
To prove the second half, observe that we can rewrite our initial equation as 
$$2^{k+1}(p+1)=2^bp +p^2 + p +2, $$
which implies that $p+1|2^bp +p^2 + p +2$. Hence $p+1|2^bp+p^2-p$, and so 
$p+1| p(2^b+p-1)$. Since $p$ and $p+1$ are relatively prime, we have then that $p+1|2^b+p-1$ Finally, we take away another multiple of $p+1$ to get $p+1|2^b-2$.
\end{proof}

\begin{lemma} The equation  \begin{equation} (2^{k+1}-1)(p^2+p+1)= 2^{k+1}p^2 + 2^bp +1 \label{Copy of Equation for Case III, j=1, general form for Lemma 3}\end{equation}  has no solutions where $p$ is odd, and $2 \leq b \leq k-1$.  \label{Lemma for j=1, a=0, b <k-1 no solutions}
\end{lemma}
\begin{proof} Assume we have a solution to the equation.  From Lemma \ref{first Lemma for Case III, a=0, j=1},  we choose an integer $x$ such that $x(p+1)=2^b-2$. Note that $x$ must be odd since $p+1$ is even, and $2^b-2$ is not divisible by 4.  We then have $2^b=xp+x +2$. When we substitute this back into Equation \ref{Copy of Equation for Case III, j=1, general form for Lemma 3}, and solve for $2^{k+1}$, we get that 
$2^{k+1}=xp+p+2$. By taking the difference of these two expressions, we obtain:
$$2^{k+1} -2^b = (xp+p+2) -(xp+x +2) = p-x.$$ Because $b \leq k-1$,  
$2^{k+1} -2^b  > 2^k$, we have
$p-x > 2^k$, and hence $p > 2^k+1$. But this contradicts Lemma \ref{first Lemma for Case III, a=0, j=1}, since we must have $p|2^k-1$.
\end{proof}

\begin{lemma} Let $n$ be a 2-near perfect number of the form $n=2^kp^2$ where $p$ is an odd prime. Assume further that the omitted divisors of $n$ are $d_1$ and $d_2$. Then, we have \begin{equation}  \label{fundamental equation for 2-near perfect of form power 2 times square of a  prime}
    d_1+d_2=-p^2+(2^{k+1}-1)p+(2^{k+1}-1),
\end{equation} and $d_1$ and $d_2$ are of opposite parity. \label{d1 and d2 opposite parities in 2-near of form power 2 times square of prime }
\end{lemma}
\begin{proof} Assume as given. Then we have \[\sigma(n)-d_1-d_2=2n\]
\[\sigma(2^{k}p^2)-d_1-d_2=2(2^{k}p^2)=2^{k+1}p^2\]
\[(2^{k+1}-1)(p^2+p+1)-d_1-d_2=2^{k+1}p^2\]
\[d_1+d_2=-2^{k+1}p^2+(2^{k+1}-1)(p^2+p+1).\]
which is equivalent to Equation \ref{fundamental equation for 2-near perfect of form power 2 times square of a  prime}. Since the right-hand side of  Equation \ref{fundamental equation for 2-near perfect of form power 2 times square of a  prime} is odd, $d_1$ and $d_2$  must be of opposite parity. 
\end{proof}

Let's look at the possible divisors of $n$, assuming $n =2^kp^2$. Every possible divisor can be of one of three types. Type I divisors are powers of 2, that is $d=2^a$ for some $0 \leq a \leq k$. Type II divisors are $p$ or $p^2$. That is $d=p^m$ where $m \in \{1,2\}$. Type III divisors are of the form $d=  2^{b}p^j$  where $0< b \leq k $ and $j \in \{1,2\}$.

We may then, without loss of generality, break down our situation into the following six cases depending on all the possible combinations of omitted divisor types, $d_1$ and $d_2$. 

\begin{center}
    \begin{tabular}{ |c|c|c|}
    \hline
    Case & $d_1$ & $d_2$  \\
       \hline
    1 & I & I  \\
    2 & I & II  \\
    3 & I & III  \\
    4 & II & II  \\
    5 & II & III  \\
    6 & III & III \\
    \hline
    \end{tabular}
\end{center}

We will handle each of these six cases separately. But before we do, we will observe that Cases 4 and 6 are both trivial since they require that both $d_1$ and $d_2$ of the same parity, which contradicts Lemma \ref{d1 and d2 opposite parities in 2-near of form power 2 times square of prime }. We thus only consider Cases 1, 2, 3, and 5. 

\begin{proposition} If $n$ is a 2-near perfect number of the form $n=2^kp^2$, where $p$ is an odd prime, with omitted divisors of Case 1 form, then $n=18$, and the omitted divisors are $1$ and $2$.
\label{n =18 situation}
\end{proposition}
\begin{proof} Assume we are in Case 1. In this case, both $d_1$ and $d_2$ must be distinct powers of 2. Since $d_1 +d_2$ is odd, one of the omitted divisors must be odd (and hence equal to 1). Without loss of generality, we will set $d_1 =1$, and $d_2=2^a$ where $1 \leq a \leq k$. Putting this into Equation \ref{fundamental equation for 2-near perfect of form power 2 times square of a  prime}, we get that

\begin{equation}
    p^2-(2^{k+1}-1)p-(2^{k+1}-2^a-2)=0. \label{Case I equation form}
\end{equation}

Equation \ref{Case I equation form} is a quadratic equation in $p$. Thus, in order to have a solution, its discriminant, defined as 
\[D=2^{2k+2}+2^{k+2}-2^{a+2}-7\]
must be a perfect square. From Lemma \ref{Case I Lemma 1}, $D$ is only a perfect square if $k=a=1$. In this case, Equation  \ref{Case I equation form} becomes just $p^2-3p=0$. Thus, one must have $p=3$, and so $n=18$ with $d_1=1, d_2 = 2$. One can verify this result: $2(18)=\sigma(18)-(1+2)$.

\end{proof}

\begin{proposition} There are no 2-near perfect numbers of the form $2^kp^2$ with omitted divisors of the Case 2.
\end{proposition}
\begin{proof} We will apply the discriminant sandwich method to this situation. Again we are working with $n=2^kp^2$, but we now have omitted divisors of the following form
$ d_1=2^a$, $d_2=p^m$ where $a \in (0,k] $ and  $m \in [1,2].$  We will break Case 2 down into two subcases, depending on whether $m=1$ or $m=2$.\\

We first consider the situation where $m=1$. Applying this to Equation \ref{fundamental equation for 2-near perfect of form power 2 times square of a  prime}, we obtain

\begin{equation}
 0=p^2-(2^{k+1}-2)p-(2^{k+1}-2^a-1). \label{Case II fundamental for m=1}
\end{equation}

Equation \ref{Case II fundamental for m=1} has an even discriminant $D$, which means that if $D$ is a perfect square, it must be divisible by 4. Thus,  we can define \[ D' = \frac{D}{4}=2^{2k} -2^a,\]  and just as well assume that $D'$ is a perfect square. Note that $D'$ is still even, so we can skip over checks against odd squares. We have that 
$2^{2k} -2^a < (2^k)^2$, and so \[ 2^{2k} -2^a \le (2^k-2)^2.\] With a little algebra we then obtain that
\[2^{k+2}-2^a\leq 4\]
Without loss of generalization, let's say $k=a+m, m\in [0,k), a>0$. We then have that
\[2^a(2^{m+2}-1) =2^{a+m+2}-2^{a}\leq 4.\]
It is evident after plugging in the minimum values for $a$ and $m$ that no solution exists. We thus have shown that when $m=1$, no solution exists.
\\
\\
We now consider the case when $m=2.$  We then obtain from  Equation \ref{fundamental equation for 2-near perfect of form power 2 times square of a  prime}, \begin{equation}
0=2p^2-(2^{k+1}-1)p-2^{k+1}+2^a+1. \label{Case II m=2, form of fundamental equation}\end{equation}

We then need that the discriminant D, defined as
\[ D=2^{2k+2}-2^{k+2}+2^{k+4}-2^{a+3}-7.\]
is a perfect square. We thus must have \[D<(2^{k+1}+3)^2.\]

Since $D$ is odd, it cannot be equal to the next smallest square, which is even. So, \[D \leq (2^{k+1}+1)^2.\]

We thus have \[2^{2k+2}-2^{k+2}+2^{k+4}-2^{a+3}-7\le 2^{2k+2}+2^{k+2}+1,\]

which implies that \[ 2^k-2^a \le 1.\]

Thus, we can only have a solution when $k=a$ or we have $k=1$ and $a=0$. However, since $d_1$ and $d_2$ must be of opposite parity, we cannot have $a=0$. Thus, we need consider only the case when $k=a$. Our expression for $D$ simplifies so that we have $D=2^{2k+2} +2^{k+2}-7 $. But this quantity cannot be a perfect square since 
$$(2^{k+1})^2 < 2^{2k+2} +2^{k+2}-7 < (2^{k+1}+1)^2, $$ and so $D$ is again sandwiched between two consecutive perfect squares. Thus, there are no solutions for Case 2 when $m=2$. Since no solutions exist for all possible cases for $m$, Proposition 9 is proven.
\end{proof}

\begin{lemma} If $p$ is an odd number such that \begin{equation} (2^{k+1}-1)(p^2+p+1)=2^{k+1}p^2+2^b p^2 + 1 \label{Equation for Case III, j=2, general form first copy}
\end{equation} where $k$ and $b$ are positive integers, then $p|2^{k}-1$. \label{p divides Case III lemma}
\end{lemma}
\begin{proof}
    Assume one has a solution to Equation \ref{Equation for Case III, j=2, general form first copy}. Then, if we take the equation modulo $p$, we get that 
\[2^{k+1}-1 \equiv 1  \mod p,\]
and hence $p|2^{k+1}-2= 2(2^{k}-1)$. Since $p$ is odd, we have $p|2^{k}-1$.
\end{proof}

\begin{lemma}
    If $p$ is an odd number which is a solution to Equation \ref{Equation for Case III, j=2, general form first copy}, then $p+1|2^b+2$\label{p+1 divides Case III lemma}.
\end{lemma}

\begin{proof}
    Assume one has a solution to Equation \ref{Equation for Case III, j=2, general form first copy}. We can rewrite this as 
    \[2^{k+1}(p+1)=(2^b+1)p^2+p+2\]
    which implies that 
    \[p+1|(2^b+1)p^2+p+2,\] and hence
    \[p+1|(2^b+1)p^2-p = p((2^b+1)p-1).\]
    Since $p$ and $p+1$ are relatively prime,
    \[p+1|(2^b+1)p-1,\]
    and so by similar logic, 
    \[p+1|2^b+2.\]
\end{proof}
Note that Lemma \ref{p+1 divides Case III lemma} is distinct from Lemma \ref{first Lemma for Case III, a=0, j=1}, since the equations needed are different, and one has a $+2$, and the other a $-2$ on the right-hand side. 

\begin{proposition}  Let $n =2^kp^2$ be a 2-near perfect number with omitted divisors of the Case 3 form with omitted divisors $1$ and $2^b p^2$. Then $n=36$, and our omitted divisors are $1$ and $18.$ 
\end{proposition}
\begin{proof} 
Assume as given. So from Equation \ref{fundamental equation for 2-near perfect of form power 2 times square of a  prime}, we have  some $b$ such  $b \leq k $, and $p$ prime such that  \begin{equation}(2^{k+1}-1)(p^2+p+1)=2^{k+1}p^2+2^b p^2 + 1\label{Case III p=2 equation}.\end{equation}

By Lemma \ref{p divides Case III lemma} and \ref{p+1 divides Case III lemma}, we have $$p|2^{k}-1$$ and  $$p+1|2^b+2.$$
Thus, there is a positive integer $z$ such that \[ z(p+1)=2^b+2,\] and

\[ z(p+1)-2=2^b.\]

\noindent If we take Equation \ref{Case III p=2 equation} modulo $2^b$ we also get that $2^b|p^2+p+2$.

We also have

\[z(p+1)-2|p^2+p+z(p+1)\]
\[z(p+1)-2|p(p+1)+z(p+1)\]
\begin{equation}
 z(p+1)-2|(p+1)(p+z). \label{z(p+1)-2 weaker divis}
\end{equation}

Let $Q$ be some integer such that $Q|z(p+1)-2$ and $Q|(p+1)$. Then Q will divide any linear combination of those terms. Thus we have 
\[ Q|z(p+1)-2-z(p+1) =-2.\]

Thus, the only possible common factors of $z(p+1)-2$ and $(p+1)$ are 1 and 2. Hence Equation \ref{z(p+1)-2 weaker divis} may be strengthened to: 

\begin{equation}z(p+1)-2|2(p+z) \label{z division equation}\end{equation}
Therefore, we know that
\[ z(p+1)-2 \le 2(p+z),\]
which implies that 
\[ z \le \frac{2p+2}{p-1} = 2 +\frac{4}{p-1}. \]

Since $p\geq 3$, we have $z \leq 3$, and hence have only three cases, $z=1$, $z=2$ or $z=3$.  One can easily check that if $p=3$ then the only case which leads to integer values is when $z=1$ and $b=1$. Here $n=36$, and our omitted divisors are $d_1=1, d_2=18$. Thus, we may assume that  $p>3$, which implies $z=1$ or $z=2$. However, if $b=1$, we get a contradiction if $p> 3$. Thus,  may assume that $b>1$ which forces  $z$ to be odd, and so $z=1$.

We have from Equation \ref{z division equation}, $p-1 | 2(p+1)$. So there is some $m$ such that $m(p-1) = 2(p+1)$. If $m=1$ then we get a negative value for $p$ and if $m=2$, we immediately get a contradiction. So we may assume that $m \geq 3$. If $m=3$, then we have $3(p-1)=2(p+1)$, which yields $p=5$ which quickly leads to a contradiction. We thus must have $m \geq 4$. However, if $4(p-1) \leq 2(p+1) $, then one must have $p=3$, but we are in the situation where $p>3$. 

Thus, our only possibility is when $n=36$.    
\end{proof}

\begin{proposition} If $n$ is a 2-near perfect number of the form $n=2^kp^2$, where $p$ is an odd prime, with omitted divisors of Case V form, then $n=200$.
\end{proposition}
\begin{proof} Assume we have such an $n$, with omitted divisors $d_1$ and $d_2$. Then, without loss of generality, we may assume that $d_1=p^j$ for $j\in[1,2]$ and $d_2=2^bp^g$ for some  $g\in[1,2]$, and $ 1\leq b\leq k$. We may break this down into four cases:

\begin{center}
    \begin{tabular}{ |c|c|c| }
    \hline
    Case & $d_1$ & $d_2$ \\
       \hline
    1 & $p$ & $2^bp$ \\
    2 & $p$ & $2^bp^2$ \\
    3 & $p^2$ & $2^bp$ \\
    4 & $p^2$ & $2^bp^2$ \\
    \hline
    \end{tabular}
\end{center}

Case 1 can be handled by the discriminant sandwich technique. 

For Case 1, Equation \ref{fundamental equation for 2-near perfect of form power 2 times square of a  prime} becomes

\begin{equation}  p+2^bp=-p^2+(2^{k+1}-1)p+(2^{k+1}-1) .   \label{V-1}
\end{equation}

The relevant discriminant from Equation \ref{V-1} is $$D= 2^{2k+2} -2^{k+b+2} +2^{2b} +2^{b+2}.$$

We have then $$(2^{k+1}-2^b)^2 < D<  (2^{k+1}-2^b+1)^2, $$

so $D$ cannot be a perfect square. Thus the equation has no solutions.
\\
Now, consider Case 2. In this case Equation \ref{fundamental equation for 2-near perfect of form power 2 times square of a  prime} becomes

\begin{equation}  p+2^bp^2=-p^2+(2^{k+1}-1)p+(2^{k+1}-1).    \label{V-2}
\end{equation}

The relevant discriminant from Equation \ref{V-2} is 

\begin{equation} D = 2^{2k+2} +2^{k+b+3} -2^{b+2}. \label{repaired case 2 discriminant.} \end{equation} 

We wish to show that such a $D$ is not a perfect square. We do so by splitting into two subcases, when $b$ is even and when $b$ is odd. In both cases we will get a contradiction from assuming that $D$ is a perfect square.

First, assume that $b$ is even. Then $2^{b+2}$ is a perfect square, and so if $D$ is a perfect square, we may factor out $2^{b+2}$ from Equation \ref{repaired case 2 discriminant.}. In that case, we have then
that $$D_0 = 2^{2k-b} +2^{k+1} -1$$ must be a perfect square. However, if $b < k$, then $D_0 \equiv 3$ (mod 4), and thus $D_0$ cannot be a pefect square. Thus, we must have $b=k$. In that situation we have $D_0= 2^k +2^{k+1}-1$. However, since $b$ is even and greater than $1$, we must then have $k \geq 2$. Thus,  we still have $D_0 \equiv 3$ (mod 4). 

Now, for our second subcase, assume that $b$ is odd. Then $2^{b+1}$ is a perfect square, and so we may factor that quantity out of $D$ and still have a perfect square. We thus have that $D_1 = 2^{2k-b+1} + 2^{k+2} -2$  must be a perfect square. However, we have then that $D_1 \equiv 2$ (mod 4), and so we have again reached a contradiction.

Since both cases lead to a contradiction, we conclude that the relevant discriminant is never a perfect square, and thus the equation has no solutions.

We now consider Case 3, where $d_1=p^2$ and $d_2=2^bp^2$.  Equation \ref{fundamental equation for 2-near perfect of form power 2 times square of a  prime} then becomes: 

\[p^2+2^bp=-p^2+(2^{k+1}-1)p+2^{k+1}-1,\]
and hence 
\begin{equation}
    -2p^2+(2^{k+1}-2^b-1)p+2^{k+1}-1=0.  \label{Case V-3}
\end{equation}

This has a corresponding discriminant value given by 
\begin{equation}
    D=x^2-2x(2^b-3)+2^{2b}+2^{b+1}-7, \label{other D value eq needing a label}
\end{equation}

where $x=2^{k+1}$. We note that if $b=1$, then we get that either $p=-1$ or $p=\frac{2^{k+1}-1}{2}$, neither of which is a prime. Thus, we may assume that $b> 1$.
Since $b > 1,$ we have 
\[2^{b+3}>16,\] which implies that
\[2^{b+1}-7>-6\cdot 2^b+9.\] We then obtain that
\[x^2-2x(2^b-3)+2^{2b}+2^{b+1}-7>x^2-2x(2^b-3)+2^{2b}-6\cdot 2^b+9,\]
which implies that
\[D>x^2-2x(2^b-3)+(2^b-3)^2 = (x-(2^b-3))^2.\]

If we have a solution to our original equation, $D$ must be a perfect square. Equation \ref{other D value eq needing a label}  also shows that $D$ must be odd. Thus, we cannot have $D= (x-(2^b-2))^2,$
and thus we have
\begin{equation}D \geq (x-(2^b-1))^2 \label{Lower bound for D} \end{equation}

\begin{equation}
    3\cdot 2^b\geq 2^{k+1}+8,
\end{equation}
\begin{equation}
    4\cdot 2^b> 2^{k+1}. \label{4 2 to b ineq}
\end{equation}

Inequality \ref{4 2 to b ineq} implies that $b > k-1$. Since we have that $b \leq k$, and $b$ is a natural number, we conclude that $b=k$. We thus may replace $k$ with $b$ in Equation \ref{Case V-3} to obtain
\[p^2+2^bp=-p^2+(2^{b+1}-1)p+2^{b+1}-1,\]
which is equivalent to 
\begin{equation}2^b(p+2)= 2p^2+p+1.\end{equation}
Thus, we have
$p+2|2p^2+p+1$. We then have:
\[p+2\vert (2p^2+p+1) + (3-2p)(p+2) = 7\]
Since $p+2|7$, we must have $p=5$. We then can solve to get that $b=k=3$. This yields 
$n=(2^3)(5^2)=200$, which is in fact a 2-near perfect number of the desired form. Here our omitted divisors are $25$ and $40$.

We now turn our attention to Case 4. That is, we have  $d_1=p^2$ and $d_2=2^bp^2$. In this situation Equation \ref{fundamental equation for 2-near perfect of form power 2 times square of a  prime} becomes

$$(2^{k+1}-1)(p^2+p+1)-2^{k+1}p^2=p^2+2^bp^2,$$

which can be rewritten as
\begin{equation}
    2^{k+1}(p+1)=2p^2+2^bp^2+p+1. \label{Case V-4}
\end{equation}

Therefore, we have: 
\[(p+1)\vert(2p^2+2^bp^2+p+1),\] 
\[(p+1)\vert(2p^2+2^bp^2+p+1)-(p+1)=2p^2+2^bp^2,\] and thus
\[(p+1)\vert p^2(2+2^b).\]
Since $p+1$ and $p^2$ are relatively prime this implies
\[(p+1)\vert(2+2^b).\]

Thus, there exists a positive integer $z$ such that $z(p+1)=2+2^b,$ 
and hence
\begin{equation}
2^b= z(p+1)-2.\label{2 to be the in terms of p and z}
\end{equation}

We note that  $p+1$ is even and the only way which  $2+2^b$ can be divisible by 4 is if $b=1$ (which does not lead to a solution). Thus, $z$ is odd. 

If we take Equation \ref{Case V-4} modulo $2^b$,
we obtain that

\begin{equation}2^b|2p^2+p+1. \label{V-4 2 power divisible function of p}\end{equation}

We note that Equation \ref{V-4 2 power divisible function of p} allows one to obtain a finite set of possible values $b$ for any given fixed choice of $p$, and then use each to solve for $k$.   We may thus with only a small amount of effort verify that we must have $p > 23$.

We may combine Equation \ref{V-4 2 power divisible function of p} with Equation \ref{2 to be the in terms of p and z} to obtain:

\[z(p+1)-2\vert2p^2+p+1.\]

We then have:
 \[zp+z-2\vert2p^2+p+1\]
 \[zp+z-2\vert2z(2p^2+p+1)-(4p-2)(zp+z-2)=4z+8p-4\]


and so 
\begin{equation}
    zp+z-2\vert4(z+2p-1).\label{V-4 divisility before ineq} 
\end{equation}

Consider now the situation where we have equality in the relationship in Equation \ref{V-4 divisility before ineq}. Then we have 

 $$zp+z-2=4(z+2p-1),$$

which is equivalent to $z(p+1)=8p+2$. Thus, 
$$p+1|8p+2,$$ and hence $$p+1|6p.$$ Since $p+1$ and $p$ are relatively prime, this forces us to have $p+1|6$, and hence we must have $p=5$, which we can verify does not work. Thus, we must have some integer $m \geq 2$ such that $$m(zp+z-2)=4(z+2p-1),$$
and thus we have
$$zp +z-2 \leq 2(z+2p-1),$$
which is equivalent to 
\begin{equation}z \leq\frac{4p}{p-1}. \label{Final z inequality}\end{equation}

We have that $p \geq 7$, and thus, Inequality \ref{Final z inequality} implies that $z < 6$. Since $z$ is odd, we must then have $z=1$, $z=3$ or $z=5$. If we have $z=1$, then Equation \ref{V-4 divisility before ineq} implies that 
$p-1\vert8p.$ But since $p-1$ is relatively prime to $p$,  we must have $p-1\vert8,$
which is impossible since we know that $p>23$.
\\
\\
Using similar logic,  for $z=3$, we obtain from Equation \ref{V-4 divisility before ineq} that
\[3p+1\vert4(2p+2)=8(p+1)\]
\[3p+1\vert8p+8-8(3p+1)=-16p.\] Thus,
$3p+1$ is relatively prime to $p$, so we must have that 
\[3p+1\vert16.\]
But, once again, we have $p>23$, so there cannot be any solution. 
\\

Finally, for $z=5$, we obtain, from Equation \ref{V-4 divisility before ineq}, that
\[5p+3\vert4(2p+4)=8p+16,\] and thus
\[5p+3\vert3(8p+16)=24p+48\]
\[5p+3\vert24p+48-16(5p+3)=-56p.\]
Since $p\neq3$, $5p+3$ and $p$ are relatively prime, so we must have
\[5p+3\vert56.\]
But, once again this contradicts that $p>23$.

\end{proof}

Thus, we have completed the proof of Theorem  \ref{Second main result power of 2 times square of a prime}.

\section{Strongly 2 near perfect numbers}

A slightly different way of defining a number to be pseudoperfect is to say that a number $n$ is pseudoperfect if there is a set $S$ which is a subset of the positive divisors of $n$ such that the sum of the elements in $S$ sums to $2n$.  The last author and Tim McCormack \cite{MZ} studied what they called strongly pseudoperfect numbers. A number $n$ is said to be strongly pseudoperfect if there is a subset $S$ of divisors of $n$ where the sum of the elements sums to $2n$, and where we also have the property that $d \in S $ if and only if $\frac{n}{d} \in S$.  It is natural to combine the notion of 2-near perfect and strongly pseudoperfect as follows: We say that a number $n$ is strongly 2-near perfect if $n$ is strongly pseudoperfect and also 2-near perfect. Note that this is equivalent to $n$ having a divisor  $d$ such that $$\sigma(n)-d -\frac{n}{d}= 2n.$$ 

The following table gives all seven strongly 2-near perfect numbers less than one million:
\begin{center}
\begin{tabular}{ |c|c|c|c| } 
 \hline
$n$ & $\sigma(n)$ & $d_1$  & $d_2$   \\ 
156 & 392 & 2 & 78 \\
352 & 756 & 8 & 44 \\
6832 & 15376 & 4 & 1708 \\
60976 & 122512 & 148 & 412 \\
91648 & 184140 & 128 & 716 \\
152812 & 306432 &  302 & 506 \\
260865 & 539136 & 15 & 17391 \\
\hline
\end{tabular}
\end{center}

In this section, we will give a description of all strongly 2-near perfect numbers $n$ of the form $n=2^kp$ for a prime $p$.

\begin{lemma} If $n$ is a strong 2-near perfect number of the form $2^kp$ for some odd prime $p$ and natural number $k$, then $p=\frac{2^{k+1}-2^a-1}{1+2^{k-a}}$. \label{Strong near perfect only arise from family 3}
\end{lemma}
\begin{proof} Assume we have a strong 2-near perfect number. By looking at our four families of numbers which arise from Theorem \ref{classification of 2 near perfect of form power of two times a prime}, we can see that only numbers in the third family might possibly be strongly 2-near perfect. In the first family, the product of omitted divisors $d_1d_2$ is odd, so one cannot have $d_1d_2=n$ since $n$ is even. In the second family, we have $d_1d_2$ is a power of 2, and thus is not $n$. In our fourth family, we have $p^2|d_1d_2$ so $d_1d_2 \neq n$. 

Thus, we may assume that we have a number arising from the third family. In that situation, from $d_1d_2=n$ we get that $a+b=k$, from which the result follows.
\end{proof}

\begin{lemma} Assume that $n$ is a strong $2$-near perfect number of the form $n=2^kp$ with $p=\frac{2^{k+1}-2^a-1}{1+2^{k-a}}$. Then $k <2a$. \label{k < 2a lemma}
\end{lemma}
\begin{proof} Assume that  $n$ is a strong $2$-near perfect number of the form $n=2^kp$ with $p=\frac{2^{k+1}-2^a-1}{1+2^{k-a}}$, and that $k \geq 2a$. Thus, we have
$$1+2^{k-a}|2^{k+1}-2^a-1,$$

which implies that 
$$1+2^{k-a}|2^{k+1}-2^a-1 +(1+2^{k-a}) = 2^{k+1} -2^a  + 2^{k-a} = 2^a(2^{k+1-a} + 2^{k-2a}-1). $$
Since  $k \geq 2a$, $2^{k+1-a} + 2^{k-2a}-1$  is a positive integer. We also have that $(1+2^{k-a}, 2^a)=1$, so we have then that 
\begin{equation}1+2^{k-a}|2^{k+1-a} + 2^{k-2a}-1 \label{Right before we break down into two cases}.\end{equation}

Consider the situation where $k=2a$. Then Equation \ref{Right before we break down into two cases}, becomes 

$$1+2^{a}|2^{a+1}, $$
which has no solutions. So, we may assume that $k> 2a$.

We have from Equation \ref{Right before we break down into two cases} that there is some $m$ such that \begin{equation}\label{m form} m(1+2^{k-a}) = 2^{k+1-a} + 2^{k-2a}-1 .\end{equation} Note that the right-hand side of the equation is odd, so $m$ must be odd. If $m=1$ then we have 
$$(1+2^{k-a}) = 2^{k+1-a} + 2^{k-2a}-1, $$
which implies that 
\begin{equation}\label{m =1}1 + 2^{k-a-1} = 2^{k-a} -2^{k-2a-1}. \end{equation} 

The left-hand side of Equation \ref{m =1} is odd, and the only way for the right-hand side to be odd is if $k-2a-1=0$. The only solution of this system of equations is when $k=3$ and $a=1$, which forces $p=\frac{13}{5}$ which is not an integer. Thus, we have $m \neq 1$, and so $m \geq 3$.

We thus have $3(1 + 2^{k-a-1}) \leq 2^{k-a} -2^{k-2a-1}$, which is impossible.

\end{proof}

\begin{proposition} Assume that $n$ is a strong $2$-near perfect number of the form $n=2^kp$ with $p=\frac{2^{k+1}-2^a-1}{1+2^{k-a}}$. Then $k =a+2$, and the omitted divisors are $d_1 =2^{a}$ and $d_2 =4p$, with $p=\frac{2^{a+3}-2^a-1}{5}$, and $a \equiv 3 $ (mod 4). 
\label{description of strongly 2 near perfect of form power of 2 times a prime}
\end{proposition}

First, we need to prove the following lemma. 

\begin{lemma} If $2^b+1|3(2^a)+1$, then $b=2$. \label{lemma about 2 b+1|3(2 a)+1} 
\end{lemma}
\begin{proof} Assume that $2^b+1|3(2^a)+1$. We thus have for some positive integer $m$, \begin{equation}
m(2^b+1)=3(2^a)+1 \label{m form for 2 powers lemma}. \end{equation}  Notice that $3(2^a)+1$ is never divisible by 3, and thus $b$ must be even, and $m$ cannot be divisible by 3. If $m=1$ then  we have $2^b+1=3(2^a)+1$ which would imply we would have $3|2^b$, which cannot happen. Thus, we may assume that $m \geq 5$. This implies that $a >b$. 

Set $a=bq +r$ where $0 \leq r  < b$.

We have \begin{equation} (3)(2^{qb+r} +1) = (3)((2^b)^q 2^r +1) \equiv 3((-1)^q 2^r +1) \pmod {2^b+1} \equiv 0  \label{q congruence I} . \end{equation}

Now, we separate into two cases, depending on whether $q$ is even or odd. If $q$ is even, then Equation \ref{q congruence I} yields that 

\begin{equation} 3(2^r)+1 \equiv 0 \pmod {2^b+1},
\end{equation}

and so

\begin{equation} 2^b +1 |3 (2^r) +1 \label{consequence case even of q congruence}
\end{equation} 

Equation \ref{consequence case even of q congruence} implies that $r \geq b-1$. We thus have $r=b-1$. Thus, 
$$2^b +1 | 3(2^{b-1}) +1 = 2^{b} +1 +2^{b-1}   $$
which is impossible. 

We now consider the case where $q$ is odd. Then Equation \ref{q congruence I} implies that 

\begin{equation} -3(2^r)+1 \equiv 0 \pmod {2^b+1}. 
\end{equation}

and thus, 
$$2^b +1| 3(2^r)-1,$$ which similarly leads to a contradiction, unless $b=2$ and $r=1$.
 
 \end{proof}   

One might wonder if Lemma \ref{lemma about 2 b+1|3(2 a)+1} can be strengthened to conclude that if $p$ is a prime where $p|2^b+1$ for some even $b$ and $p|3(2^a) +1$ for some $a$, then one must have $p=5$. However, this is in fact not true. In particular, note that $29|2^{14}+1$, but it is also true that $29|3(2^9)+1. $

We now prove Proposition \ref{description of strongly 2 near perfect of form power of 2 times a prime}.
 \begin{proof} Assume that  $n$ is a strongly $2$-near perfect number of the form $n=2^kp$ with $p=\frac{2^{k+1}-2^a-1}{1+2^{k-a}}$. Our proof is complete if we can show that we must have $k=a+2$.  If we have $k=a+b$, then this is the same as $2^b +1|2^{a+b+1} -2^a -1$, which implies that 
 $2^b +1|2^{a+b+1} -2^a +2^b$. 
We have
$$2^b+1 |2^{a+b+1} -2^a -1 -2^{a+1}(1+2^b) = -3(2^a)-1 $$
and so $2^b +1|3(2^a)+1$, which allows us to apply Lemma \ref{lemma about 2 b+1|3(2 a)+1}, to conclude that $b=2$, and the rest follows simply from noting that all 2-near perfect of this form are in Case 3. 
 \end{proof}

We list below the first few values of $a$ where $\frac{2^{a+3}-2^a-1}{5}$ is prime, and its corresponding prime $p$, each of which corresponds to a  strong 2-near perfect number of the form $2^{a+2} p$. We do not include the last two primes as they are too big to fit on one line.

   \begin{center}
\begin{tabular}{ |c|c| } 
 \hline
$a$ & $p$ \\
3 & 11 \\
7 & 179 \\
19 & 734003 \\
27 &187904819 \\
31 &3006477107 \\
39 &769658139443 \\
151 &3996293539576687666963200714458586381871690547 \\
1 99 & -- \\
451 & -- \\
\hline
\end{tabular}
\end{center}

Standard heuristic arguments suggest that there should be infinitely many primes of the form $\frac{2^{a+3}-2^a-1}{5}$. 

\section{Open problems}

One obvious direction is try to extend the classification of 2-near perfect numbers to classify all of the form $2^kp^m$ where $m\geq3$. 

\begin{conjecture} There are only finitely many 2-near perfect numbers of the form $2^kp^m$ where $m\geq2$.
\end{conjecture}

A slightly weaker conjecture is the following.

\begin{conjecture} For any fixed $m \geq2$, there are only finitely many 2-near perfect numbers of the form $2^kp^m$.
\end{conjecture}

Another direction to go in is to change the signs in the relationship $\sigma(n)=2n+d_1 +d_2$. The two other options are $\sigma(n)=2n-d_1 -d_2$ and $\sigma(n)=2n+d_1 -d_2$. It seems likely that the main method used in this paper, including the discriminant sandwich would be successful for the first of these two situations, but the situation with mixed signs on the divisors may be more difficult.

\section{Acknowledgements}
This paper was written as part of the Hopkins School Mathematics Seminar 2022-2023. Steven J. Miller suggested the problem of changing signs as discussed in the final section.

\end{document}